\documentclass[11pt]{article}
\usepackage[T1]{fontenc}
\usepackage{lmodern}
\usepackage[margin=1in]{geometry}
\usepackage{amsmath,amsfonts,amssymb,amsthm}
\usepackage{microtype}
\usepackage[authoryear,round]{natbib}
\usepackage[colorlinks=true,linkcolor=blue,citecolor=blue,urlcolor=blue]{hyperref}
\numberwithin{equation}{section}
\theoremstyle{plain}
\newtheorem{theorem}{Theorem}[section]
\newtheorem{assumption}[theorem]{Assumption}
\newtheorem{corollary}[theorem]{Corollary}
\newtheorem{lemma}[theorem]{Lemma}
\theoremstyle{remark}
\newtheorem{remark}[theorem]{Remark}
\usepackage{bm}
\usepackage{bbm}
\usepackage{xcolor}
\usepackage{mathtools}
\newcommand{\norm}[1]{\lVert #1\rVert}
\newcommand{\ip}[2]{\langle #1,#2\rangle}
\newcommand{\Prob}{\mathbb P}
\DeclareMathOperator{\diag}{diag}
\DeclareMathOperator{\tr}{tr}

\DeclareMathOperator{\Var}{Var}

\DeclareMathOperator{\spec}{spec}
\DeclareMathOperator{\dist}{dist}
\newcommand{\E}{\mathbb{E}}

\newcommand{\op}{\mathrm{op}}

\title{Distributional Limits for Eigenvalues of Graphon Kernel Matrices}
\author{Behzad~Aalipur\thanks{Department of Mathematical Sciences, University of Cincinnati.
Email: \href{mailto:aalipubd@mail.uc.edu}{\texttt{aalipubd@mail.uc.edu}}}}
\date{}
\hypersetup{pdftitle={Distributional Limits for Eigenvalues of Graphon Kernel Matrices},pdfauthor={Behzad Aalipur}}
\begin{document}
\maketitle
\begin{abstract}
We study eigenvalue fluctuations of kernel and adjacency matrices in dense graphon-based random graph models, for modes associated with simple, isolated nonzero eigenvalues of the graphon integral operator. Our results require no smoothness or Lipschitz continuity. They include degenerate non-leading modes whose normalized eigenfunctions have constant modulus and change sign between $-1$ and $+1$ on measurable sets. Continuous graphons exclude this possibility: their nonzero eigenfunctions have continuous representatives, so constant modulus forces a constant eigenfunction, which corresponds to the largest eigenvalue of a nonnegative graphon. Nevertheless, non-leading degeneracy already occurs at the second eigenvalue of a two-block model with equal blocks.

The limiting behavior is governed by the variance of the squared population eigenfunction. When this variance is positive, the empirical eigenvalue, divided by $n-1$ and centered at its population value, has a Gaussian limit under $\sqrt{n}$ scaling. When the variance vanishes, the first-order fluctuation disappears, and the centered unnormalized kernel eigenvalue converges to a weighted series of centered chi-square variables.

For adjacency matrices, Bernoulli edge noise is negligible at the $\sqrt{n}$ scale of the unnormalized eigenvalues. In the degenerate regime, after the same deterministic centering used for the kernel eigenvalue, the adjacency eigenvalue converges to the same weighted chi-square series plus a deterministic shift and a Gaussian term independent of that series. The proofs combine spectral localization, exact Schur identities, finite-mode approximations, and Hoeffding decompositions.
\end{abstract}

\section{Introduction}

Spectral methods play a central role in the analysis of random graphs and network data 
\citep{lei2015consistency,rohe2011spectral}. Work on community detection 
\citep{abbe2017community,lei2015consistency,zhao2012consistency}, clustering 
\citep{von2007tutorial}, and network-model estimation \citep{bickel2011method,gao2015rate} 
has motivated a substantial body of work on the asymptotic behavior of eigenvalues and 
eigenvectors in latent-variable models. Related statistical inference for adjacency eigenvalues in weighted networks is developed by \citet{aalipur2026bootstrap}, who propose a bootstrap method and establish its consistency. Classical work establishes spectral convergence of sampled kernel matrices
\citep{koltchinskii2000random}, and quantitative eigenvalue approximation
bounds are developed by \citet{braun2006accurate}. In a different spectral
regime, \citet{zhu2024central} establish central limit theorems
for global linear spectral statistics of inhomogeneous random-graph
adjacency matrices with graphon limits.
However, the distributional behavior of individual eigenvalues, a question of central 
importance for uncertainty quantification remains comparatively underdeveloped.

This paper develops a distributional theory for simple eigenvalues of kernel matrices generated by dense graphon-based random graphs. Working within the graphon framework \citep{lovasz2006limits,lovasz2012large,borgs2008convergent,borgs2012convergent}, we show that each empirical eigenvalue exhibits a sharp probabilistic dichotomy in its fluctuations. Depending on the degeneracy 
structure of the population eigenfunction, the leading fluctuations are either asymptotically 
Gaussian or converge to an explicit weighted chi-square law. \citet[Theorems 2.1--2.2]{chatterjee2025fluctuation}
establish this dichotomy for the largest kernel and adjacency eigenvalues
under Lipschitz regularity. Their Remark 2.2 discusses extensions to other
isolated kernel eigenvalues, and Remark 2.6 addresses non-leading adjacency
eigenvalues. Our results remove the Lipschitz assumption in the bounded
measurable setting $W\in L^2([0,1]^2)\cap L^\infty([0,1]^2)$ and use
$\Var(\varphi_r(U)^2)=0$ as the degeneracy criterion. This setting permits
discontinuous sign-valued eigenfunctions with $\varphi_r^2=1$, allowing
order-one degenerate fluctuation laws for non-leading isolated modes.
It also covers discontinuous and non-smooth models, including stochastic
block models with unequal block sizes.
The results 
demonstrate that smoothness assumptions, while technically convenient, are not intrinsic to the 
limiting behavior.

From a probabilistic perspective, the fluctuation behavior is driven by the interaction between 
two sources of randomness: latent-position sampling and Bernoulli edge noise. A central finding 
of this work is that, at the $\sqrt{n}$ scale, the dominant contribution arises from the latent 
variables, while edge noise is asymptotically negligible. This separation, implicit in earlier 
work on spectral convergence \citep{koltchinskii2000random} and latent-position 
models \citep{lei2015consistency,abbe2017community,athreya2018statistical}, becomes explicit in 
our analysis and clarifies the sources of uncertainty in spectral procedures. In particular, the 
zero diagonal of the kernel matrix-an intrinsic structural feature of graphon models-plays a 
decisive role in the degenerate regime, where the linear Hoeffding projection vanishes and the 
limiting distribution becomes a weighted chi-square series.

Our proof combines three ingredients. Finite-rank Gram convergence and a sampled Hilbert--Schmidt tail bound localize the target eigenvalue. An exact Schur identity at that empirical eigenvalue controls the compression resolvent. Finally, finite-sample residual and canonical U-statistic tail identities permit joint finite-mode truncation and a converging-together argument. Hoeffding decomposition supplies the first-order normal limit. For adjacency matrices, conditional matrix Bernstein bounds and an explicit characteristic-function argument identify the additional order-one edge contribution. The leading-mode edge constants agree with Proposition 4.5 and Theorem 2.2 of \citet{chatterjee2025fluctuation}.

The remainder of the paper is organized as follows. Section~\ref{sec:prelim} introduces the 
graphon model, the kernel matrix, and the associated integral operator. Section~\ref{sec:main} 
states the main fluctuation theorem and outlines the perturbation-theoretic approach. 
Section~\ref{sec:examples} presents non-Lipschitz examples illustrating the scope of the theory. 
Section~\ref{sec:discussion} discusses implications for spectral methods and directions for future 
work. 
Technical proofs are given in Section~\ref{sec:tech} and the appendices.

\section{Preliminaries}
\label{sec:prelim}

A graphon is a measurable symmetric function
\begin{equation}
W : [0,1]^2 \to [0,1], \quad W(x,y)=W(y,x) \text{ for almost every } (x,y)\in[0,1]^2.
\label{eq:graphon_def}
\end{equation}
Graphons provide a canonical framework for modeling dense inhomogeneous random graphs and arise as limits of sequences of finite graphs
\citep{lovasz2006limits,lovasz2012large,borgs2008convergent,borgs2012convergent}.
To generate a random graph on $n$ vertices, we first sample latent variables
\begin{equation}
U_1,\dots,U_n \stackrel{\mathrm{i.i.d.}}{\sim} \mathrm{Unif}[0,1].
\label{eq:latent_vars}
\end{equation}
Conditional on $(U_1,\dots,U_n)$, edges are generated independently according to
\begin{equation}
A_{ij} \mid U_1,\dots,U_n \;\sim\; \mathrm{Bernoulli}\bigl(W(U_i,U_j)\bigr), \quad i<j,
\label{eq:edge_generation}
\end{equation}
with $A_{ji}=A_{ij}$ and $A_{ii}=0$.
This defines a simple undirected random graph whose conditional expectation satisfies
\begin{equation}
\mathbb{E}\bigl[\mathbf{A}_n \mid U_1,\dots,U_n\bigr] = \mathbf{K}_n,
\label{eq:conditional_mean}
\end{equation}
where $\mathbf{K}_n$ denotes the kernel matrix defined below
\citep{bickel2009nonparametric,wolfe2013nonparametric,chatterjee2015matrix}.

The kernel matrix is given by
\begin{equation}
(\mathbf{K}_n)_{ij} :=
\begin{cases}
W(U_i,U_j), & i\neq j,\\
0, & i=j.
\end{cases}
\label{eq:kernel_matrix}
\end{equation}
The zero diagonal reflects the absence of self-loops and is intrinsic to graphon-based network models.
This structural restriction plays a nontrivial role in the fluctuation behavior, particularly in degenerate regimes.
As will be seen later, it has important consequences for the asymptotic behavior of spectral statistics.

The graphon $W$ induces a compact self-adjoint integral operator
\begin{equation}
T_W : L^2([0,1]) \to L^2([0,1]), \quad (T_W f)(x) := \int_0^1 W(x,y)f(y)\,dy.
\label{eq:integral_operator}
\end{equation}
Since $W\in L^2([0,1]^2)$, the operator $T_W$ is Hilbert-Schmidt and therefore compact.

Choose an arbitrary enumeration $(\lambda_k,\varphi_k)$ of its nonzero eigenvalues, counting multiplicities, with orthonormal eigenfunctions. For an empirical target $\lambda_r\ne0$, the notation $\lambda_r(\mathbf K_n)$ denotes the unique isolated branch near $n\lambda_r$ with probability tending to one, not an ordering by absolute value. On exceptional events assign arbitrary measurable values. The eigenfunctions are normalized according to

\begin{equation}
\mathbb{E}[\varphi_k(U)^2]=1, \quad U\sim\mathrm{Unif}[0,1].
\label{eq:eigenfunction_normalization}
\end{equation}
For each fixed $k$, the empirical eigenvalues of the kernel matrix satisfy
\begin{equation}
\lambda_k(\mathbf{K}_n)/(n-1)\xrightarrow{\mathbb{P}}\lambda_k
\quad\text{as }n\to\infty,
\label{eq:eigenvalue_convergence}
\end{equation}
See \citet{koltchinskii2000random}; for the matching simple nonzero branch used here, Lemma~\ref{kr:lem:localize} gives the required localization and convergence. 

The corresponding population eigenfunction $\varphi_r$ is discretized as
\begin{equation}
\mathbf{v}_r := (\varphi_r(U_1),\dots,\varphi_r(U_n))^\top\in\mathbb{R}^n.
\label{eq:unnormalized_eigenvector}
\end{equation}
Define the normalization factor
\begin{equation}
s_{r,n} := \sum_{i=1}^n\varphi_r(U_i)^2 = n\bigl(1+V_{r,n}\bigr),
\label{eq:normalization_factor}
\end{equation}
where
\begin{equation}
V_{r,n} := \frac{1}{n}\sum_{i=1}^n\bigl(\varphi_r(U_i)^2-1\bigr).
\label{eq:centered_eigenfunction_variance}
\end{equation}
The normalized discretization vector is
\begin{equation}
\mathbf{u}_r := s_{r,n}^{-1/2}(\varphi_r(U_1),\dots,\varphi_r(U_n))^\top,
\label{eq:normalized_eigenvector}
\end{equation}

which satisfies $\|\mathbf u_r\|_2=1$ when $s_{r,n}>0$. The event $s_{r,n}=0$ has probability tending to zero; make arbitrary unit-vector choices there.

Let $A$ be a self-adjoint operator with a simple eigenvalue $\lambda_r$ and normalized eigenvector $\psi_r$, and let $E$ be a self-adjoint perturbation. The second-order Rayleigh-Schr\"odinger expansion yields
\[
\lambda_r(A+E)-\lambda_r(A)
=
\langle\psi_r,E\psi_r\rangle
+
\sum_{k\neq r}\frac{|\langle\psi_k,E\psi_r\rangle|^2}{\lambda_r-\lambda_k}
+
R(E),
\]

For a finite symmetric matrix, writing $e=\|E\|_{\op}$ and $g$ for the simple-eigenvalue gap, the second-order remainder satisfies $|R(E)|\le4e^3/g^2$ when $e\le g/4$. If the displayed quadratic term is omitted, the first-order remainder is bounded by $2e^2/g$; see Lemma~\ref{lem:expansion}.

In the present setting, the leading term corresponds to the discrete Rayleigh quotient
\begin{equation}
\mathbf{u}_r^\top\mathbf{K}_n\mathbf{u}_r
=
\frac{1}{s_{r,n}}
\sum_{i\neq j}\varphi_r(U_i)W(U_i,U_j)\varphi_r(U_j).
\label{eq:rayleigh_quotient}
\end{equation}

The unnormalized numerator, divided by $n(n-1)$, is a second-order U-statistic; the Rayleigh quotient additionally contains the random normalization $s_{r,n}^{-1}$. Its symmetric kernel is

\begin{equation}
h_r(x,y) := \varphi_r(x)W(x,y)\varphi_r(y).
\label{eq:ustat_kernel}
\end{equation}
The Hoeffding decomposition gives
\begin{equation}
\frac{1}{n(n-1)}\sum_{i\neq j}h_r(U_i,U_j)
=
\theta
+
\frac{2}{n}\sum_{i=1}^n h_1(U_i)
+
U_n^{(2)},
\label{eq:hoeffding_decomposition}
\end{equation}
where
\(
\theta := \mathbb{E}[h_r(U,U')]\), \( h_1(x) := \mathbb{E}[h_r(x,U)]-\theta\), and
\(U_n^{(2)}\) is a degenerate second-order U-statistic.

We defer discussion of how the linear and degenerate Hoeffding components determine the Gaussian versus weighted chi-square regimes to Section~\ref{sec:proof_strategy}.

\section{Main Results}
\label{sec:main}

Recall from Section~\ref{sec:prelim} that $\mathbf{K}_n$ denotes the kernel matrix with zero diagonal,
and $\mathbf{A}_n$ is the adjacency matrix with conditional mean $\mathbb{E}[\mathbf{A}_n \mid U_1,\dots,U_n] = \mathbf{K}_n$.

To isolate the leading-order fluctuations of a target eigenvalue $\lambda_r$, we work under a small 
set of regularity conditions on the kernel $W$ and its spectrum. These assumptions are very mild and 
designed to ensure that the integral operator $T_W$ has a simple, well-separated eigenvalue with an eigenfunction having finite fourth moment that are natural in the graphon setting and hold for a broad class 
of non-Lipschitz kernels.

\begin{assumption}\label{ass:ustat}
Fix an index $r\ge1$. Let $W\in L^2([0,1]^2)\cap L^\infty([0,1]^2)$ be symmetric, 
and let $T_W:L^2([0,1])\to L^2([0,1])$ denote the associated integral operator
\[
(T_W f)(x) := \int_0^1 W(x,y) f(y)\,dy.
\]
Let $(\lambda_k, \varphi_k)_{k\ge1}$ denote an $L^2$-orthonormal eigendecomposition of $T_W$, 

using the fixed enumeration above. Assume:

\begin{enumerate}
\item (\textbf{Spectral gap at $r$}) $\lambda_r\neq 0$ is simple and
\[
\gamma_r:=\dist(\lambda_r,\spec(T_W)\setminus\{\lambda_r\})>0.
\]
\item (\textbf{Bounded eigenfunctions}) No additional moment assumption is needed: for every nonzero mode,
$\|\varphi_k\|_\infty\le\|W\|_\infty/|\lambda_k|$.
\end{enumerate}
\end{assumption}

\begin{remark}

The eigenrelation gives
$|\varphi_k(x)|\le|\lambda_k|^{-1}\|W\|_\infty\|\varphi_k\|_{L^1}
\le|\lambda_k|^{-1}\|W\|_\infty$ almost everywhere. Thus all finite moments used below exist. Since $0\in\spec(T_W)$, $\gamma_r\le|\lambda_r|$.

\end{remark}

\begin{theorem}[Unified eigenvalue fluctuation dichotomy]
\label{thm:ustat}
Under Assumption~\ref{ass:ustat}, let $\lambda_r$ denote the $r$th eigenvalue of the 
population integral operator $T_W$. Then:

\begin{enumerate}
\item \textbf{Non-degenerate case.}
If $\sigma_r^2 := \Var(\varphi_r(U)^2) > 0$, then
\[
\sqrt{n}\left(\frac{\lambda_r(\mathbf{K}_n)}{n-1}-\lambda_r\right)
\xrightarrow{d}
\mathcal{N}(0,\lambda_r^2\sigma_r^2).
\]

\item \textbf{Degenerate case.}
If $\sigma_r^2 := \Var(\varphi_r(U)^2) = 0$, which holds if and only if 
$\varphi_r^2 \equiv 1$ almost everywhere, then
\[
\lambda_r(\mathbf{K}_n) - (n-1)\lambda_r - C_r
\xrightarrow{d}
\sum_{k\neq r}\frac{\lambda_r\lambda_k}{\lambda_r-\lambda_k}(Z_k^2-1),
\qquad
C_r:=\sum_{k\neq r}\frac{\lambda_k^2}{\lambda_r-\lambda_k}.
\]
where $(Z_k)_{k\ge1}$ are i.i.d.\ standard normal random variables, and the series 

converges in $L^2$ and almost surely. The deterministic series $C_r$ converges absolutely because $|\lambda_k^2/(\lambda_r-\lambda_k)|\le\lambda_k^2/\gamma_r$. Absolute convergence of the centered random series is not asserted.
\end{enumerate}

\end{theorem}

\begin{remark}[Origin of the degenerate coefficients]
The coefficient $\frac{\lambda_r \lambda_k}{\lambda_r - \lambda_k}$ is the algebraic sum of 
two distinct sources of fluctuation.
Note that 
\[
\lambda_k(Z_k^2-1)
+
\frac{\lambda_k^2}{\lambda_r-\lambda_k}(Z_k^2-1)
=
\frac{\lambda_r\lambda_k}{\lambda_r-\lambda_k}(Z_k^2-1),
\]
Both terms are $O_p(1)$ in the degenerate regime.

\begin{enumerate}
    \item The \textbf{linear term} contributes $\lambda_k(Z_k^2-1)$.
    
This is the centered canonical Rayleigh contribution of the hollow kernel matrix; even though

    $\mathbb{E}[h_{r,2}(U_1,U_2)\mid U_1]=0$, the 
    degenerate U-statistic $\frac{1}{n}\sum_{i\neq j}h_{r,2}(U_i,U_j)$ does not vanish.
    \item The \textbf{resolvent term} contributes 
    
$\frac{\lambda_k^2}{\lambda_r-\lambda_k}Z_k^2$. Its expectation contributes to $C_r$; its centered part shares the same $Z_k$ as the Rayleigh term. The two components are not independent.

\end{enumerate}

\end{remark}

\subsection{Proof strategy}
\label{sec:proof_strategy}

We work directly with the sampled matrix rather than perturbing its unconditional expectation. Put $t_n=\lambda_r(\mathbf K_n)$, $\lambda=\lambda_r$, $\phi=\varphi_r$, and let $u_n=n^{-1/2}(\phi(U_i))_i$. Lemma~\ref{kr:lem:localize} localizes $t_n$ and controls alignment. In the degenerate regime $\|u_n\|=1$, and with $Q_n=I-u_nu_n^{\mathsf T}$ the exact identity is
\begin{equation}\label{eq:rs_outline}
t_n=u_n^{\mathsf T}K_nu_n+
(Q_ng_n)^{\mathsf T}(t_nI-Q_nK_nQ_n)^{-1}(Q_ng_n),
\qquad g_n=(K_n-n\lambda I)u_n,
\end{equation}
where $K_n=\mathbf K_n$ and the inverse is on $u_n^\perp$. Lemma~\ref{kr:lem:schur} proves the necessary separation of this non-invariant compression. The center is the empirical eigenvalue $t_n$.

For $\sigma_r^2>0$, normalization of the sampled Rayleigh quotient reduces the first Hoeffding coefficient from $2\lambda_r$ to $\lambda_r$, giving the stated variance. When $\phi^2=1$, both the canonical Rayleigh term and the Schur correction have order one. The exact residual-tail identity in Lemma~\ref{kr:lem:tail} and the U-statistic tail variance permit joint truncation; the two finite-mode components use the same empirical coordinates
$\xi_{n,k}=n^{-1/2}\sum_i\phi(U_i)\varphi_k(U_i)$. Adding their joint limits gives the coefficients in Theorem~\ref{thm:ustat}. Full details are in Section~\ref{sec:tech} and Appendix A.

\subsection{Implications for the adjacency matrix}
When analyzing the adjacency matrix $\mathbf{A}_n$ with conditional mean $\mathbb{E}[\mathbf{A}_n\mid U_1,\dots,U_n]=\mathbf{K}_n$, 
the eigenvalue fluctuations at the $\sqrt{n}$ scale are governed by the latent kernel $\mathbf{K}_n$, not by the Bernoulli edge noise.

In the notation $\lambda=\lambda_r$, $\phi=\varphi_r$, let
$\gamma=\dist(\lambda,\sigma(T_W)\setminus\{\lambda\})>0$.
In this note $\phi^2$ need not be constant.
For iid uniform latent positions $U_i$, set
\[
(K_n)_{ij}=W(U_i,U_j)1_{i\ne j},\qquad
(A_n)_{ij}\mid\mathcal U_n\sim\operatorname{Bernoulli}(W(U_i,U_j)),\quad i<j,
\]
independently, with symmetric entries and zero diagonals, where
$\mathcal U_n=\sigma(U_1,\ldots,U_n)$. Write $E_n=A_n-K_n$ and
\[
\begin{gathered}
s(x,y)=W(x,y)(1-W(x,y)),\quad\tau=\iint\phi(y)^2s(x,y)\,dx\,dy,\\
\nu=2\iint\phi(x)^2\phi(y)^2s(x,y)\,dx\,dy.
\end{gathered}
\]
All integrals are Lebesgue integrals over $[0,1]^2$. By symmetry $\tau$ also
equals the integral weighted by $(\phi(x)^2+\phi(y)^2)/2$. Both constants are
finite and nonnegative; $\nu$ denotes a variance, not a standard deviation.
Let $t_n$ be the isolated branch of $K_n$ near $n\lambda$, and $v_n$ a unit
eigenvector for it. Choose all branches and vectors measurably; on bad events
make arbitrary measurable choices. Such events have probability tending to zero.

\begin{theorem}[Order-one edge-noise correction]\label{er:thm:edge}
With probability tending to one, there is exactly one eigenvalue $\widehat t_n$
of $A_n$ within $\gamma n/4$ of $t_n$, and
\begin{equation}\label{er:eq:edge-expansion}
\widehat t_n-t_n=v_n^{\mathsf T}E_nv_n+\tau/\lambda+o_p(1).
\end{equation}
For every real $t$,
\begin{equation}\label{er:eq:conditional}
\E[e^{it(\widehat t_n-t_n)}\mid\mathcal U_n]
\to_p\exp(it\tau/\lambda-t^2\nu/2).
\end{equation}
Thus $\widehat t_n-t_n\Rightarrow N(\tau/\lambda,\nu)$. When $\nu=0$, this
means a point mass at $\tau/\lambda$, which need not be zero.
If $H_n$ is $\mathcal U_n$-measurable and $H_n\Rightarrow H$, then
$(H_n,\widehat t_n-t_n)\Rightarrow(H,G)$, with
$G\sim N(\tau/\lambda,\nu)$ independent of $H$.
\end{theorem}

\begin{corollary}[Adjacency eigenvalue fluctuations in the non-degenerate regime]
\label{cor:edge_noise}
Let $\mathbf{A}_n$ be the adjacency matrix with $\mathbb{E}[\mathbf{A}_n\mid U_1,\dots,U_n]=\mathbf{K}_n$.
Assume Assumption~\ref{ass:ustat} and suppose $\sigma_r^2:=\Var(\varphi_r(U)^2)>0$.
Then
\[
\sqrt{n}\left(
\lambda_r\!\left(\frac{\mathbf{A}_n}{n-1}\right)
-
\lambda_r\!\left(\frac{\mathbf{K}_n}{n-1}\right)
\right)
\overset{p}{\longrightarrow} 0,
\]
and
\[
\sqrt{n}\left(
\lambda_r\!\left(\frac{\mathbf{A}_n}{n-1}\right)
-
\lambda_r
\right)
\xrightarrow{d}
\mathcal{N}\!\left(0,\lambda_r^2\sigma_r^2\right).
\]
\end{corollary}

\begin{proof}

Theorem~\ref{er:thm:edge} gives $\lambda_r(\mathbf A_n)-\lambda_r(\mathbf K_n)=O_p(1)$, which is negligible after division by $\sqrt n$.

The second display follows by Slutsky's theorem, combining the first display with the non-degenerate conclusion of Theorem~\ref{thm:ustat} for $\lambda_r(\mathbf{K}_n/(n-1))$.
\end{proof}

\begin{corollary}[Adjacency eigenvalue fluctuations in the degenerate regime]
\label{cor:edge_noise_degenerate}
Assume Assumption~\ref{ass:ustat}, the Bernoulli graphon model of
Theorem~\ref{er:thm:edge}, and $\varphi_r^2=1$ almost everywhere. Set
\[
\bar s=\iint_{[0,1]^2}W(x,y)(1-W(x,y))\,dx\,dy,
\qquad C_r=\sum_{k\ne r}\frac{\lambda_k^2}{\lambda_r-\lambda_k}.
\]
For the matching isolated adjacency eigenvalue branch,
\[
\lambda_r(\mathbf A_n)-(n-1)\lambda_r-C_r
\Rightarrow\sum_{k\ne r}\frac{\lambda_r\lambda_k}{\lambda_r-\lambda_k}(Z_k^2-1)
+\frac{\bar s}{\lambda_r}+\sqrt{2\bar s}\,Z_0,
\]
where $(Z_k)_{k\ne r}$ and $Z_0$ are independent standard normal variables.
The centered series converges in $L^2$ and almost surely, as in
Theorem~\ref{thm:ustat}; absolute convergence of the centered series is not asserted.
\end{corollary}
\begin{proof}
Under $\varphi_r^2=1$, the edge constants in Theorem~\ref{er:thm:edge}
reduce to $\tau=\bar s$ and $\nu=2\bar s$. The centered kernel statistic
is $\mathcal U_n$-measurable and converges by Theorem~\ref{thm:ustat}(2).
The joint-limit assertion of Theorem~\ref{er:thm:edge} makes its limit
independent of the Gaussian edge contribution. Addition proves the claim.
\end{proof}
\begin{remark}
A coarse $o_p(\sqrt n)$ perturbation estimate cannot transfer the order-one
kernel law unchanged. If $\bar s=0$, then $W\in\{0,1\}$ almost everywhere
and $\mathbf A_n=\mathbf K_n$ almost surely. For a rank-one constant graphon
$W=p$ with $0<p\le1$, the corollary gives $N(1-p,2p(1-p))$ for the
centered leading adjacency eigenvalue. The kernel law is non-Gaussian
when at least one other nonzero mode contributes; in the rank-one case
its centered random series is zero.
\end{remark}

\section{Examples}
\label{sec:examples}

This section presents several examples of graphons that satisfy Assumption~\ref{ass:ustat} and therefore fall within the scope of Theorem~\ref{thm:ustat}, while lying outside the Lipschitz framework considered in earlier work. These examples illustrate that $\sqrt{n}$-scale eigenvalue fluctuations depend primarily on spectral and integrability properties rather than smoothness, and they clarify the distinction between the non-degenerate and degenerate regimes.

\subsection{Brownian square-root kernel}

Define the graphon $W(x,y):=\tfrac12(\min\{x,y\}+\sqrt{xy})$ on $[0,1]^2$, and write $W=(K_0+K_1)/2$ with

$K_0(x,y):=\min\{x,y\}$ (Brownian covariance kernel, $\lambda_k(K_0)\asymp k^{-2}$)
and $K_1(x,y):=\sqrt{xy}$ (rank one). Hence $T_W$ is a rank-one perturbation of $\tfrac12 T_{K_0}$.
The min--max principle gives, for $k\ge2$,
\[
\tfrac12\lambda_k(K_0)\le\lambda_k(W)\le\tfrac12\lambda_{k-1}(K_0),
\]
so $\lambda_k(W)\asymp k^{-2}$.

The kernel is not Lipschitz: $\partial_x\sqrt{xy}=\sqrt{y}/(2\sqrt{x})$ diverges as
$x\to0$, and the same local obstruction implies it is not piecewise Lipschitz under
any finite interval partition.

Nevertheless, $W\in L^2([0,1]^2)\cap L^\infty([0,1]^2)$ 
with $\|W\|_\infty\le1$.

The leading eigenvalue is positive and simple: $W>0$ almost everywhere,
so a sign-changing maximizing eigenfunction would have a strictly smaller
Rayleigh quotient than its absolute value. A nonnegative leading
eigenfunction is strictly positive almost everywhere by the eigenrelation;
two such orthogonal eigenfunctions cannot exist. Since $W$ is continuous,
its nonzero eigenfunctions have continuous representatives. The leading
eigenfunction satisfies $\varphi_1(0)=0$, which rules out
$\varphi_1^2=1$ almost everywhere by continuity. Thus the leading mode
satisfies Assumption~\ref{ass:ustat} and is non-degenerate.
For the leading mode, and for any other simple nonzero mode with $\sigma_r^2>0$,
Theorem~\ref{thm:ustat} yields
$\sqrt{n}(\lambda_r(\mathbf{K}_n)/(n-1)-\lambda_r)\xrightarrow{d}\mathcal{N}(0,\lambda_r^2\sigma_r^2)$
whenever $\sigma_r^2>0$, showing coverage beyond Lipschitz kernels.

\subsection{Stochastic block model with unequal block sizes}

Consider a two-block stochastic block model with block proportions $\pi_1=1/3$ and $\pi_2=2/3$, 
within-block edge probability $p$, and cross-block probability $0<q<p\le1$.

 The associated graphon is
\[
W(x,y)=
\begin{cases}
p, & (x,y)\in[0,1/3]^2\cup[1/3,1]^2,\\
q, & \text{otherwise}.
\end{cases}
\]
This kernel is piecewise constant and discontinuous. The second eigenfunction $\varphi_2$ encodes the block structure. 
For $0<q<p\le1$, $\varphi_2$ is non-constant and has larger absolute value on the smaller block.

 The population matrix in the orthonormal block-indicator basis is
\[
\begin{pmatrix}p/3&\sqrt2q/3\\\sqrt2q/3&2p/3\end{pmatrix},
\qquad \lambda_{1,2}=\frac{3p\pm\sqrt{p^2+8q^2}}6.
\]
Both eigenvalues are positive and distinct. Writing $d=\sqrt{p^2+8q^2}$,
the block values $c_1,c_2$ of the second eigenfunction satisfy
$c_2/c_1=(p-d)/(4q)=-2q/(p+d)$, whose absolute value is strictly less
than $1/2$ when $0<q<p$. In particular their squared values differ.
After normalization, $\sigma_2^2=c_1^4/3+2c_2^4/3-1>0$.
Consequently,
$\mathbb{E}[\varphi_2(U)^2]=1$ and $\Var(\varphi_2(U)^2)>0$, placing this model in the non-degenerate regime.

The kernel satisfies Assumption~\ref{ass:ustat}. It belongs to $L^2([0,1]^2)$ and $L^\infty([0,1]^2)$, its spectrum is finite dimensional and determined by the block structure, and its eigenfunctions are bounded and piecewise constant. Theorem~\ref{thm:ustat} therefore yields
\[
\sqrt{n}\left(\frac{\lambda_2(\mathbf{K}_n)}{n-1}-\lambda_2\right)
\xrightarrow{d}
\mathcal{N}\bigl(0,\lambda_2^2\sigma_2^2\bigr),
\]
where $\sigma_2^2=\Var(\varphi_2(U)^2)>0$. This example illustrates that the theory applies to practically relevant network models that fall outside Lipschitz-based analyses.

\subsection{Symmetric stochastic block model}

For a symmetric two-block stochastic block model with equal block sizes $\pi_1=\pi_2=1/2$, within-block probability $p$, and cross-block probability $0<q<p\le1$, the graphon is
\[
W(x,y)=
\begin{cases}
p, & (x,y)\in[0,1/2]^2\cup[1/2,1]^2,\\
q, & \text{otherwise}.
\end{cases}
\]
By symmetry, the second eigenfunction satisfies $\varphi_2(x)\in\{-1,1\}$ almost everywhere, so that $\varphi_2^2\equiv1$ and $\Var(\varphi_2(U)^2)=0$. This model therefore lies in the degenerate regime of Theorem~\ref{thm:ustat}. The eigenvalue fluctuation is non-Gaussian and satisfies
\[
\lambda_2(\mathbf K_n)-(n-1)\lambda_2-C_2
\xrightarrow{d}
\sum_{k\neq2}\frac{\lambda_2\lambda_k}{\lambda_2-\lambda_k}(Z_k^2-1),
\]
where $\lambda_1=(p+q)/2$, $\lambda_2=(p-q)/2$, and
$C_2=\lambda_1^2/(\lambda_2-\lambda_1)$; the limit series converges in $L^2$ under the Hilbert--Schmidt condition. The empirical convergence is in distribution.

Corollary~\ref{cor:edge_noise_degenerate} also gives the adjacency law.
Here
\[
\bar s=\frac{p(1-p)+q(1-q)}2,\qquad
C_2=-\frac{(p+q)^2}{4q},
\]
and hence
\[
\lambda_2(\mathbf A_n)-(n-1)\lambda_2-C_2
\Rightarrow -\frac{p^2-q^2}{4q}(Z_1^2-1)
+\frac{p(1-p)+q(1-q)}{p-q}
+\sqrt{p(1-p)+q(1-q)}\,Z_0,
\]
with $Z_0,Z_1$ independent. For example, at $(p,q)=(4/5,1/5)$,
$\lambda_2=3/10$, $C_2=-5/4$, the chi-square coefficient is $-3/4$,
the edge shift is $8/15$, and the Gaussian variance is $8/25$.

 Despite its structural similarity to the unequal-block model, symmetry induces a fundamentally different limiting distribution, illustrating the dichotomy underlying Theorem~\ref{thm:ustat}.

\subsection{\texorpdfstring{$\alpha$-H\"older kernels}{alpha-Holder kernels}}

Consider the power-law kernel $W(x,y)=(xy)^\alpha$ for $\alpha\in(0,1)$. The kernel satisfies Assumption~\ref{ass:ustat}. It belongs to $L^2([0,1]^2)$ since
\[
\int_0^1\int_0^1(xy)^{2\alpha}\,dx\,dy
=
\left(\int_0^1 x^{2\alpha}\,dx\right)^2
=
\frac{1}{(2\alpha+1)^2}<\infty,
\]
and it belongs to $L^\infty([0,1]^2)$ with $\|W\|_\infty=1$. 
The associated operator is compact, positive semidefinite and rank one. Its unique nonzero eigenpair is $\lambda_1=(2\alpha+1)^{-1}$ and $\varphi_1(x)=\sqrt{2\alpha+1}\,x^\alpha$, which is continuous and bounded.

The kernel is not Lipschitz. Its partial derivatives satisfy
\[
\frac{\partial W}{\partial x}(x,y)=\alpha y^\alpha x^{\alpha-1},
\]
which diverges as $x\to0$ for $\alpha<1$. Consequently, a Lipschitz bound of the form
\[
|x^\alpha-x'^\alpha|\le L|x-x'|
\]
cannot hold uniformly on $[0,1]$. Its fourth moment gives
\[
\sigma_1^2=\frac{(2\alpha+1)^2}{4\alpha+1}-1
=\frac{4\alpha^2}{4\alpha+1},\qquad
\lambda_1^2\sigma_1^2=\frac{4\alpha^2}{(4\alpha+1)(2\alpha+1)^2}>0.
\]
The leading eigenfunction $\varphi_1$ is strictly positive almost everywhere and non-constant, so that $\Var(\varphi_1(U)^2)>0$. Theorem~\ref{thm:ustat} therefore applies in the non-degenerate regime, yielding a Gaussian central limit theorem for $\lambda_1(\mathbf{K}_n)/(n-1)$.

\section{Discussion}
\label{sec:discussion}

This work develops a distributional theory for individual eigenvalues of dense graphon-based random graphs under minimal $L^2([0,1]^2) \cap L^\infty ([0,1]^2)$ assumptions. Our main result extends the eigenvalue fluctuation dichotomy established by \citet{chatterjee2025fluctuation} beyond Lipschitz kernels and shows that smoothness, while technically convenient in prior analyses, is not intrinsic to the limiting behavior described in Theorem~\ref{thm:ustat}. The key probabilistic mechanism is instead the degeneracy structure induced by the population eigenfunction and the zero-diagonal kernel-matrix normalization.

A central conceptual takeaway is that, at the $\sqrt{n}$ scale, eigenvalue fluctuations are driven primarily by sampling variability in the latent positions rather than by Bernoulli edge noise. In the non-degenerate regime, where $\sigma_r^2=\Var(\varphi_r(U)^2)>0$, this leads to a Gaussian limit with variance $\lambda_r^2\sigma_r^2$, as in Theorem~\ref{thm:ustat}. In the degenerate regime, where $\varphi_r^2\equiv 1$ almost everywhere, the $\sqrt{n}$-scale fluctuations vanish and the centered eigenvalue converges to an explicit weighted chi-square series determined by the spectrum of $T_W$. In stochastic block models, this dichotomy cleanly separates unequal-block settings (typically non-degenerate) from symmetric equal-block settings (degenerate), indicating that the non-Gaussian limit reflects an underlying structural symmetry that is invisible to first-order fluctuations.

From an inferential perspective, the results clarify the sources of uncertainty in spectral procedures for network analysis. In the non-degenerate regime, the asymptotic variance of $\lambda_r(\mathbf K_n)/(n-1)$ is $\lambda_r^2\sigma_r^2/n$, suggesting plug-in estimation strategies for confidence intervals and hypothesis tests for eigenvalue-based methods in latent-position and block models. Moreover, Corollary~\ref{cor:edge_noise} shows that, after normalization, Bernoulli edge noise is asymptotically negligible at the $\sqrt{n}$ scale, so the same distributional limits apply to the observed adjacency matrix $\mathbf A_n$. The degenerate regime offers complementary diagnostic value: detecting $\varphi_r^2\equiv 1$ corresponds to identifying symmetry in the generating structure and can inform model checking and selection.

Several extensions merit further investigation. The approach should extend to joint fluctuations of finitely many simple eigenvalues in the non-degenerate regime, yielding multivariate Gaussian limits with explicit covariance structure. Beyond eigenvalues, combining perturbation expansions with Davis-Kahan-type control may yield fluctuation results for eigenvectors and derived quantities relevant to spectral clustering uncertainty quantification. Sparse and semi-dense regimes, where the graphon or its scaling depends on $n$, require different concentration tools and may exhibit phase transitions in eigenvalue behavior. Additional challenges arise when the target eigenvalue has multiplicity greater than one, in which case perturbation analysis must be carried out at the eigenspace level and the limiting laws are expected to involve weighted quadratic forms over that eigenspace. Finally, it would be natural to adapt the framework to other graph-associated operators, such as normalized Laplacians, though controlling randomness introduced by degree normalization presents further technical difficulties.

Overall, removing smoothness assumptions provides both a technical extension and a conceptual clarification. The persistence of the fluctuation dichotomy under $L^2([0,1]^2) \cap L^\infty ([0,1]^2)$ conditions shows that the limiting distributions are determined by spectral structure and degeneracy rather than kernel regularity. The examples in Section~\ref{sec:examples} highlight that discontinuous block models, non-Lipschitz H\"older kernels, and Brownian-type covariance structures fall within the scope of the theory, substantially enlarging the class of models for which eigenvalue uncertainty quantification is available.

\section{Proof of the main theorem}
\label{sec:tech}
This section contains the proof of Theorem~\ref{thm:ustat}. We follow the strategy outlined in Section~\ref{sec:proof_strategy}: first reduce the eigenvalue fluctuation to a Rayleigh quotient (up to a negligible error) using perturbation theory for simple eigenvalues, and then analyze the resulting quadratic form via a Hoeffding decomposition. The argument splits into two regimes depending on whether the linear Hoeffding projection is nontrivial ($\sigma_r^2>0$) or vanishes identically ($\sigma_r^2=0$).

\begin{proof}[Proof of Theorem~\ref{thm:ustat}]
We treat the non-degenerate and degenerate cases separately, using spectral perturbation theory and U-statistic decompositions.

\textbf{Case 1: Non-degenerate case ($\sigma_r^2 > 0$)} Recall the normalized discretization of eigenfunction and the corresponding scaling as defined in \eqref{eq:normalized_eigenvector} and \eqref{eq:normalization_factor}:,  
\[
\mathbf{u}_r := s_{r,n}^{-1/2}(\varphi_r(U_1),\dots,\varphi_r(U_n))^\top, \qquad
s_{r,n} := \sum_{i=1}^n \varphi_r(U_i)^2.
\]
By Lemma~\ref{lem:approx_eigvec} we have 
\[
\mathbf{K}_n \mathbf{u}_r = (n-1)\lambda_r \mathbf{u}_r + \mathbf{e}_r,\qquad
\|\mathbf{e}_r\|_2 = O_p(\sqrt{n}),\qquad
\|\mathbf{u}_r\|_2=1.
\]
Define the Rayleigh quotient $\eta_n := \mathbf{u}_r^\top \mathbf{K}_n \mathbf{u}_r$ and residual
$\mathbf{r}_n := \mathbf{K}_n\mathbf{u}_r - \eta_n\mathbf{u}_r$. Then
\[
\eta_n
=
\mathbf{u}_r^\top \mathbf{K}_n \mathbf{u}_r
=
(n-1)\lambda_r + \mathbf{u}_r^\top \mathbf{e}_r
=
(n-1)\lambda_r + O_p(\sqrt n),
\]
and
\[
\mathbf{r}_n
=
\mathbf{e}_r - (\mathbf{u}_r^\top\mathbf{e}_r)\mathbf{u}_r,\qquad
\|\mathbf{r}_n\|_2 \le 2\|\mathbf{e}_r\|_2 = O_p(\sqrt n).
\]

Lemma~\ref{kr:lem:localize} localizes the isolated branch and supplies a sample gap of order $n$. Since $\eta_n=(n-1)\lambda_r+O_p(\sqrt n)= (n-1)\lambda_r+o_p(n)$, an interval about $(n-1)\lambda_r$ with half-width a sufficiently small fixed multiple of $n$ contains exactly the target eigenvalue and contains $\eta_n$, with probability tending to one. Both distances from $\eta_n$ to its endpoints are of order $n$. The corrected Kato--Temple bounds in Lemma~\ref{lem:kato-temple}, together with $\|\mathbf r_n\|^2=O_p(n)$, give
\[
\lambda_r(\mathbf K_n)=\eta_n+O_p(1)
=\mathbf u_r^{\mathsf T}\mathbf K_n\mathbf u_r+O_p(1).
\]
Thus the original normalized Hoeffding calculation below applies with an error $o_p(\sqrt n)$.

Using the kernel representation $(\mathbf{K}_n)_{ij} = W(U_i,U_j)\mathbbm{1}\{i \neq j\}$,
\[
\mathbf{u}_r^\top \mathbf{K}_n \mathbf{u}_r
= \frac{1}{s_{r,n}} \sum_{i \neq j} \varphi_r(U_i) W(U_i, U_j) \varphi_r(U_j)
= \frac{1}{s_{r,n}} \sum_{i \neq j} h_r(U_i, U_j),
\]
where $h_r(x,y) := \varphi_r(x)W(x,y)\varphi_r(y)$. For the symmetric kernel $h_r$, the Hoeffding
decomposition (see, e.g., \citet[Section 5.3]{serfling2009approximation}) gives
\[
\frac{1}{n(n-1)}\sum_{i \neq j} h_r(U_i, U_j)
= \theta + \frac{2}{n}\sum_{i=1}^n h_1(U_i) + U_n^{(2)},
\]
with 
\begin{itemize} \item $\theta := \mathbb{E}[h_r(U,U')] = \lambda_r$ (as computed in the proof of Lemma~\ref{lem:degU}); \item $h_1(x) := \mathbb{E}[h_r(x,U)] - \theta = \lambda_r\varphi_r(x)^2 - \lambda_r = \lambda_r(\varphi_r(x)^2 - 1)$; \item $U_n^{(2)} := \frac{1}{n(n-1)}\sum_{i \neq j} h_{r,2}(U_i, U_j)$ is the degenerate (second-order) U-statistic with $\mathbb{E}[U_n^{(2)}] = 0$ and $\Var(U_n^{(2)}) = O(n^{-2})$ (by standard U-statistic variance formulas for bounded kernels). \end{itemize}  Thus   \[ \frac{1}{n(n-1)}\sum_{i \neq j} h_r(U_i, U_j) = \lambda_r + \frac{2\lambda_r}{n}\sum_{i=1}^n (\varphi_r(U_i)^2 - 1) + U_n^{(2)}. \]   Define   \[ V_{r,n} := \frac{1}{n}\sum_{i=1}^n (\varphi_r(U_i)^2 - 1), \]   so that   \[ \frac{1}{n(n-1)}\sum_{i \neq j} h_r(U_i, U_j) = \lambda_r + 2\lambda_r V_{r,n} + U_n^{(2)}. \]
By Lemma~\ref{lem:norm}, $s_{r,n} = n(1 + V_{r,n})$, hence
\[
\frac{1}{s_{r,n}} = \frac{1}{n}(1 - V_{r,n} + O_p(n^{-1})).
\]
Therefore,
\begin{align*}
\mathbf{u}_r^\top \mathbf{K}_n \mathbf{u}_r
&= \frac{n(n-1)}{s_{r,n}} \cdot \frac{1}{n(n-1)}\sum_{i \neq j} h_r(U_i, U_j) \\
&= (n-1)(1 - V_{r,n} + O_p(n^{-1})) \cdot (\lambda_r + 2\lambda_r V_{r,n} + U_n^{(2)}) \\
&= (n-1)\lambda_r + (n-1)\lambda_r V_{r,n} + (n-1)U_n^{(2)} + O_p(1),
\end{align*}
where the cross-term $2\lambda_r V_{r,n} - \lambda_r V_{r,n} = \lambda_r V_{r,n}$ yields the
variance cancellation.

Now by construction, $V_{r,n}$ is the average of i.i.d.\ centered variables
$\varphi_r(U_i)^2 - 1$ with
\[
\mathbb{E}[\varphi_r(U)^2 - 1] = 0,\qquad
\Var(\varphi_r(U)^2 - 1) = \Var(\varphi_r(U)^2) = \sigma_r^2 > 0.
\]
Hence, by the classical central limit theorem,
\[
\sqrt{n} V_{r,n} \xrightarrow{d} \mathcal{N}(0, \sigma_r^2).
\]
For the degenerate U-statistic term, $\Var(U_n^{(2)}) = O(n^{-2})$ implies
\[
(n-1)U_n^{(2)} = O_p(1) = o_p(\sqrt{n}).
\]
Combining this with the above expansion of
$\mathbf{u}_r^\top \mathbf{K}_n \mathbf{u}_r$ and the approximation
$\lambda_r(\mathbf{K}_n) = \mathbf{u}_r^\top \mathbf{K}_n \mathbf{u}_r + o_p(\sqrt n)$, we obtain
\[
\lambda_r(\mathbf{K}_n) = (n-1)\lambda_r + (n-1)\lambda_r V_{r,n} + o_p(\sqrt{n}).
\]
Dividing by $\sqrt{n}$,
\[
\frac{\lambda_r(\mathbf{K}_n) - (n-1)\lambda_r}{\sqrt{n}}
= \lambda_r \sqrt{n} V_{r,n} + o_p(1)
\xrightarrow{d} \mathcal{N}(0, \lambda_r^2 \sigma_r^2),
\]
or equivalently,
\[
\sqrt{n}\Bigl(\frac{\lambda_r(\mathbf{K}_n)}{n-1} - \lambda_r\Bigr)
\xrightarrow{d} \mathcal{N}(0, \lambda_r^2 \sigma_r^2),
\]
as claimed.

\textbf{Case 2: Degenerate case ($\varphi_r^2=1$).}
Use the notation $\lambda=\lambda_r$, $\phi=\varphi_r$, $K_n=\mathbf K_n$, $t_n=\lambda_r(\mathbf K_n)$, $u_n=n^{-1/2}(\phi(U_i))_i$, and $\psi_k=\phi\varphi_k$. Then $\|u_n\|=1$, the $\psi_k$ are orthonormal and centered for $k\ne r$, and $\xi_{n,k}=n^{-1/2}\sum_i\psi_k(U_i)$ converges jointly over finite mode sets to independent standard normals. Let $J_M$ be increasing finite sets of nonzero modes containing $r$, exhausting the fixed enumeration, and $W_M=\sum_{k\in J_M}\lambda_k\varphi_k\otimes\varphi_k$. Lemmas~\ref{kr:lem:schur}--\ref{kr:lem:quad} below justify the empirical-center Schur correction and its tail passage.

Define the bounded, symmetric, canonical $L^2$ kernel
\[
h(x,y)=\phi(x)W(x,y)\phi(y)-\lambda
=\sum_{j\ne r}\lambda_j\psi_j(x)\psi_j(y),
\]
where the spectral expansion is in $L^2([0,1]^2)$ and $\int h(x,y)dy=0$ almost everywhere. The Rayleigh quotient of the \emph{hollow} matrix obeys
\begin{equation}\label{kr:eq:ustat-def}
u_n^{\mathsf T}K_nu_n-(n-1)\lambda
=\frac1n\sum_{i\ne k}h(U_i,U_k).
\end{equation}
For fixed $M$, the truncated canonical kernel $h_M=\sum_{j\in J_M\setminus\{r\}}\lambda_j\psi_j\otimes\psi_j$ gives
\begin{equation}\label{kr:eq:ustat-modes}
\frac1n\sum_{i\ne k}h_M(U_i,U_k)
=\sum_{j\in J_M\setminus\{r\}}\lambda_j
\left(\xi_{n,j}^2-\frac1n\sum_i\psi_j(U_i)^2\right)
\ \Rightarrow\ \sum_{j\in J_M\setminus\{r\}}\lambda_j(Z_j^2-1).
\end{equation}
Independence of the $U_i$ and canonical degeneracy imply the \emph{exact} variance identity
\begin{equation}\label{kr:eq:u-tail}
\E\left[\left(\frac1n\sum_{i\ne k}(h-h_M)(U_i,U_k)\right)^2\right]
=\frac{2(n-1)}n\norm{W-W_M}_{L^2}^2.
\end{equation}
This supplies a second finite-sample tail bound. Apply the multivariate central limit theorem for the coordinates of $h_M$ \emph{jointly} with the quadratic form in Lemma~\ref{kr:lem:quad}; then apply both tail estimates and the converging-together theorem. We obtain
\[
\left(\,u_n^{\mathsf T}K_nu_n-(n-1)\lambda,\quad
 t_n-u_n^{\mathsf T}K_nu_n\,\right)
\Rightarrow
\left(\,\sum_{j\ne r}\lambda_j(Z_j^2-1),\quad
\sum_{j\ne r}\frac{\lambda_j^2}{\lambda-\lambda_j}Z_j^2\,\right).
\]
The two components share the same Gaussian coordinates; they are \emph{not independent}. Adding and rearranging terms,
\[
\lambda_j(Z_j^2-1)+\frac{\lambda_j^2}{\lambda-\lambda_j}Z_j^2
=\frac{\lambda\lambda_j}{\lambda-\lambda_j}(Z_j^2-1)
+\frac{\lambda_j^2}{\lambda-\lambda_j},
\]
proves the kernel limit Theorem~\ref{thm:ustat}(2) without a trace-class assumption.

\end{proof}

\section*{Acknowledgments}
The author thanks Mohammad S.\,M.\ Moakhar for insightful discussions and valuable suggestions.


\appendix

\section*{Appendix A: Spectral Perturbation Lemmas}
\renewcommand{\thesection}{A.\arabic{section}}
\renewcommand{\thelemma}{A.\arabic{lemma}}

Throughout this appendix, Assumption~\ref{ass:ustat} is in force. Recall that
$W\in L^2([0,1]^2)\cap L^\infty([0,1]^2)$, that $T_W$ is the associated integral
operator, and that $(\lambda_k,\varphi_k)$ denote its eigenpairs.

\begin{lemma}[Approximate eigenpair for the target index $r$]
\label{lem:approx_eigvec}
Under Assumption~\ref{ass:ustat}, define
\[
\mathbf{u}_r := s_{r,n}^{-1/2}\bigl(\varphi_r(U_1),\dots,\varphi_r(U_n)\bigr)^\top,
\qquad
s_{r,n} := \sum_{i=1}^n \varphi_r(U_i)^2 .
\]
Then
\[
\|\mathbf{K}_n \mathbf{u}_r - (n-1)\lambda_r \mathbf{u}_r\|_{2} = O_p(\sqrt{n}),
\qquad
s_{r,n} = n(1 + o_p(1)).
\]
\end{lemma}

\begin{proof}
First, since $\E[\varphi_r(U)^2]=1$ and $\{\varphi_r(U_i)^2\}_{i=1}^n$ are i.i.d.,
the law of large numbers gives
\[
s_{r,n}=\sum_{i=1}^n \varphi_r(U_i)^2 = n(1+o_p(1)).
\]

Define the unnormalized vector
\[
\tilde{\mathbf u}_r := n^{-1/2}\bigl(\varphi_r(U_1),\dots,\varphi_r(U_n)\bigr)^\top .
\]
For each $i$,
\[
\Bigl(\frac{\mathbf K_n}{n-1}\tilde{\mathbf u}_r\Bigr)_i
=
\frac{1}{\sqrt n}\cdot \frac{1}{n-1}\sum_{j\ne i} W(U_i,U_j)\varphi_r(U_j).
\]
Let
\[
\Delta_i
:=
\frac{1}{n-1}\sum_{j\ne i} W(U_i,U_j)\varphi_r(U_j)-\lambda_r\varphi_r(U_i).
\]
Then $\E[\Delta_i\mid U_i]=0$. Conditioning on $U_i$, the summands are independent in $j$ and
\[
\E\!\left[ W(U_i,U_j)\varphi_r(U_j)\mid U_i\right]
=(T_W\varphi_r)(U_i)=\lambda_r\varphi_r(U_i).
\]
Moreover,
\[
\Var\!\left(W(U_i,U_j)\varphi_r(U_j)\mid U_i\right)
\le
\E\!\left[W(U_i,U_j)^2\varphi_r(U_j)^2\mid U_i\right]
\le
\|W\|_\infty^2\,\E[\varphi_r(U)^2]
=
\|W\|_\infty^2.
\]
Hence, conditional on $U_i$,
\[
\E[\Delta_i^2\mid U_i]
=
\Var(\Delta_i\mid U_i)
=
\frac{1}{(n-1)^2}\sum_{j\ne i}\Var\!\left(W(U_i,U_j)\varphi_r(U_j)\mid U_i\right)
\le \frac{C}{n-1},
\]
for a constant $C>0$. Therefore,
\[
\E\left\|\frac{\mathbf K_n}{n-1}\tilde{\mathbf u}_r-\lambda_r\tilde{\mathbf u}_r\right\|_2^2
=
\sum_{i=1}^n \E\left[\left(\frac{1}{\sqrt n}\Delta_i\right)^2\right]
=
\frac{1}{n}\sum_{i=1}^n \E[\Delta_i^2]
\le \frac{C}{n}.
\]
Thus
\[
\left\|\frac{\mathbf K_n}{n-1}\tilde{\mathbf u}_r-\lambda_r\tilde{\mathbf u}_r\right\|_2
=O_p(n^{-1/2}),
\]
equivalently,
\[
\|\mathbf K_n\tilde{\mathbf u}_r-(n-1)\lambda_r\tilde{\mathbf u}_r\|_2
=O_p(\sqrt n).
\]
Finally, since $\mathbf u_r=\sqrt{n/s_{r,n}}\,\tilde{\mathbf u}_r=(1+o_p(1))\tilde{\mathbf u}_r$,
the same bound holds with $\mathbf u_r$ in place of $\tilde{\mathbf u}_r$.
\end{proof}

For this appendix write $\lambda=\lambda_r$, $\phi=\varphi_r$, $\gamma=\gamma_r$, $K_n=\mathbf K_n$, $t_n=\lambda_r(\mathbf K_n)$ and $M_0=\|W\|_\infty$. Write $\phi_j=\varphi_j$ for every enumerated mode. Let $J$ denote the index set of the nonzero eigenmodes in this enumeration (finite or countably infinite), and put $f_{j,n}=(\phi_j(U_i))_{i\le n}$ for $j\in J$. Use the isolated branch fixed in Section~\ref{sec:prelim}. The eigenrelation implies $\|\varphi_k\|_\infty\le M_0/|\lambda_k|$. The localization lemma is valid without degeneracy; the residual-tail and finite-mode lemmas below assume $\phi^2=1$.
\setcounter{subsection}{0}
\renewcommand{\thesubsection}{A.\arabic{subsection}}
\subsection{Localization and an exact compression argument}
Choose increasing finite subsets $J_M\uparrow J$ that contain $r$ and let
\[
W_M=\sum_{j\in J_M}\lambda_j\phi_j\otimes\phi_j,
\qquad W_{>M}=W-W_M.
\]
For fixed $M$, both $W_M$ and its diagonal $W_M(x,x)$ are bounded. Moreover, $\norm{W_{>M}}_{L^2}\to0$.
\begin{lemma}[Localization and alignment, without degeneracy]\label{kr:lem:localize}
Let $\lambda\ne0$ be a simple eigenvalue of $T_W$, with normalized real
eigenfunction $\phi$ and gap $\gamma$ to the rest of $\sigma(T_W)$, including $0$.
The hypothesis $\phi^2=1$ is not needed for this lemma. With probability tending
to one there is exactly one eigenvalue $t_n$ of $K_n$ in
$n(\lambda-\gamma/4,\lambda+\gamma/4)$, and
\[
t_n/n\to_p\lambda,\qquad
\min_{\mu\ne t_n}|t_n-\mu|\ge\gamma n/2,\qquad |t_n|\ge\gamma n/2.
\]
Put $f_n=(\phi(U_i))_{i\le n}$, $u_n=f_n/\sqrt n$, and
$w_n=f_n/\norm{f_n}$ whenever $f_n\ne0$. Then
\[
\begin{gathered}
\norm{u_n}^2=1+O_p(n^{-1/2}),\quad
\norm{(K_n-(n-1)\lambda I)u_n}=O_p(\sqrt n),\\
\norm{v_n-u_n}=O_p(n^{-1/2}),\quad
\norm{v_n}_\infty=O_p(n^{-1/2}).
\end{gathered}
\]
where $v_n$ is a unit eigenvector for $t_n$, signed to agree with $w_n$.
All choices on exceptional events may be measurable arbitrary choices.
If $\phi^2=1$, then $w_n=u_n$ exactly and
$g_n=(K_n-n\lambda I)u_n$ satisfies $\E\norm{g_n}^2\le C_0 n$.
\end{lemma}
\begin{proof}
The point $0$ belongs to the spectrum of a compact operator on the
infinite-dimensional space $L^2[0,1]$, so $\gamma\le|\lambda|$.
Let $K_{n,M}$ be the hollow sampled truncation. Independence gives
\[
\E\norm{(K_n-K_{n,M})/n}_F^2
=(1-n^{-1})\norm{W-W_M}_{L^2}^2.
\]
For fixed $M$, restoring the diagonal of $K_{n,M}$ costs $O_M(1)$ in operator
norm. Write $F_{n,M}=\sum_{j\in J_M}\lambda_j f_{j,n}f_{j,n}^{\mathsf T}$ and
$G_{n,M}=n^{-1}(\langle f_{j,n},f_{k,n}\rangle)_{j,k\in J_M}$.
The finite-dimensional LLN gives $G_{n,M}\to_p I$. On the event $G_{n,M}>0$,
the nonzero eigenvalues of $F_{n,M}/n$ equal those of
$G_{n,M}^{1/2}\operatorname{diag}(\lambda_j)G_{n,M}^{1/2}$; this follows by
orthonormalizing the sampled columns. This symmetric matrix tends to
$\operatorname{diag}(\lambda_j)$. The other $n-|J_M|$ eigenvalues are zero.

Here is the probability bookkeeping. For any fixed small $\epsilon>0$, Weyl's
inequality on $\mathbb R^n$ and Markov's inequality show that the limsup
probability of a failure to match the finite-rank spectrum within $\epsilon$
is at most $\epsilon^{-2}\norm{W-W_M}_{L^2}^2$. Every retained mode other than
$\lambda$, and every adjoined zero, has distance at least $\gamma$ from
$\lambda$. Thus an error smaller than $\gamma/8$ gives uniqueness and the
stated gap. First take $n\to\infty$ for fixed $M$, and then let $M\to\infty$;
the limsup failure probability is zero. The same argument with arbitrarily
small $\epsilon$ proves $t_n/n\to_p\lambda$. It does not claim that one fixed
nonzero tail alone gives a probability tending to one.

For almost every $U_i=x$, the $n-1$ summands in $(K_nf_n)_i$ are conditionally
independent with mean $\lambda\phi(x)$ and bounded second moments. Hence
\[
\E\norm{(K_n-(n-1)\lambda I)u_n}^2\le C_1n.
\]
Boundedness of $\phi$ also gives the LLN and the $O_p(n^{-1/2})$ normalization
error for $\norm{u_n}^2$. On the localization event, all other eigenvalues are
at distance at least $\gamma n/4$ from $(n-1)\lambda$ for sufficiently large
$n$. Expand $w_n$ in an orthonormal eigenbasis. Its residual of size
$O_p(\sqrt n)$ implies squared mass $O_p(n^{-1})$ outside $v_n$; therefore
$\norm{v_n-w_n}=O_p(n^{-1/2})$. Combine with
$\norm{w_n-u_n}=O_p(n^{-1/2})$. Finally
$\norm{v_n}_\infty\le\norm\phi_\infty/\sqrt n+\norm{v_n-u_n}_2$.
Under degeneracy $\norm{u_n}=1$ and changing the center from $(n-1)\lambda$ to
$n\lambda$ adds the bounded vector $\lambda u_n$, giving the assertion about
$g_n$.
\end{proof}

\noindent From here through the degenerate proof, $\phi^2=1$ again.
All resolvents are set to zero off their stated good events; all identities
involving them are asserted on those events, whose probabilities tend to one.

Set $Q_n=I-u_nu_n^{\mathsf T}$, $P_n=I-v_nv_n^{\mathsf T}$, $M_n=K_n/n$, and $\rho_n=t_n/n$.
\begin{lemma}[Compression gap and exact trial-vector Schur formula]\label{kr:lem:schur}
Let $D_n=Q_nK_nQ_n$ and $D'_n=P_nK_nP_n$, each acting on the full $\mathbb R^n$. Then
\[
\norm{D_n-D'_n}_{\op}=O_p(\sqrt n),\quad
\dist(t_n,\sigma(D'_n))\ge cn,\quad
\dist(t_n,\sigma(D_n))\ge\tfrac c2n
\]
with probability tending to one for some fixed $c>0$. Consequently $(t_nI-Q_nK_nQ_n)$ is invertible on $u_n^\perp$, with inverse of norm $O_p(n^{-1})$. Further,
\begin{equation}\label{kr:eq:schur}
 t_n-u_n^{\mathsf T}K_nu_n
 = (Q_ng_n)^{\mathsf T}\bigl(t_nI-Q_nK_nQ_n\bigr)^{-1}(Q_ng_n).
\end{equation}
With the \emph{zero-extended} rescaled resolvent
\[
\mathcal R_n=Q_n(\rho_nI-Q_nM_nQ_n)^{-1}Q_n,
\]
which vanishes on $\mathrm{span}(u_n)$, we have
\begin{equation}\label{kr:eq:scaled-schur}
 t_n-u_n^{\mathsf T}K_nu_n
 =(g_n/\sqrt n)^{\mathsf T}\mathcal R_n(g_n/\sqrt n),
\qquad \norm{\mathcal R_n}_{\op}=O_p(1).
\end{equation}
\end{lemma}
\begin{proof}
Write $Q_n-P_n=v_nv_n^{\mathsf T}-u_nu_n^{\mathsf T}$. Lemma~\ref{kr:lem:localize} implies $\norm{Q_n-P_n}_{\op}=O_p(n^{-1/2})$, whereas $\norm{K_n}_{\op}\le n\norm W_\infty$. Hence
\[
\norm{Q_nK_nQ_n-P_nK_nP_n}_{\op}
\le2\norm{K_n}_{\op}\norm{Q_n-P_n}_{\op}=O_p(\sqrt n).
\]
Since $P_n$ commutes with $K_n$, the full-space spectrum of $P_nK_nP_n$ consists of \emph{zero} (along $v_n$) and the remaining eigenvalues of $K_n$. Lemma~\ref{kr:lem:localize} separates $t_n$ from both sets by $cn$. Weyl's spectral perturbation bound gives the lower bound for $D_n$, including the potentially spurious zero eigenvalue. Inversion is now justified both on the full space and on $u_n^\perp$.
Relative to the orthogonal decomposition $\mathbb Ru_n\oplus u_n^\perp$, the matrix $K_n$ has block form $\bigl(\begin{smallmatrix}\alpha&b^{\mathsf T}\\b&C\end{smallmatrix}\bigr)$ with $\alpha=u_n^{\mathsf T}K_nu_n$, $b=Q_nK_nu_n=Q_ng_n$, and $C=Q_nK_nQ_n|_{u_n^\perp}$. The eigenvalue equation for $t_n$ and invertibility of $t_nI-C$ give exactly $t_n=\alpha+b^{\mathsf T}(t_nI-C)^{-1}b$. Rescaling yields \eqref{kr:eq:scaled-schur}. No approximate value $n\lambda$ or $(n-1)\lambda$ is substituted into the resolvent.
\end{proof}
\subsection{Exact finite-sample residual tail and quantitative fixed-mode control}
For $j\ne r$ let $\psi_j=\phi\phi_j$. Since $\phi^2=1$, these are mean-zero orthonormal functions. Put
\[
\xi_{n,j}=n^{-1/2}\sum_{i=1}^n\psi_j(U_i),\qquad
e_{j,n}=n^{-1/2}(\phi_j(U_1),\ldots,\phi_j(U_n))^{\mathsf T},
\]
and
\[
 g_{n,M}=\left(\sum_{j\in J_M\setminus\{r\}}
 \lambda_j\phi_j(U_i)\xi_{n,j}\right)_{i\le n}
 =\sqrt n\sum_{j\in J_M\setminus\{r\}}\lambda_j\xi_{n,j}e_{j,n}.
\]
\begin{lemma}[Exact mean-square truncation identity]\label{kr:lem:tail}
For every fixed $M$ and every $n\ge2$,
\begin{equation}\label{kr:eq:tail}
\boxed{\frac1n\E\norm{g_n-g_{n,M}}^2
=\left(1-\frac1n\right)\norm{W-W_M}_{L^2}^2
 +\frac1n\int_0^1W_M(x,x)^2dx.}
\end{equation}
In particular, for every $\varepsilon>0$,
\begin{equation}\label{kr:eq:iter-tail}
\lim_{M\to\infty}\limsup_{n\to\infty}
\Prob\bigl(\norm{g_n-g_{n,M}}>\varepsilon\sqrt n\bigr)=0.
\end{equation}
\end{lemma}
\begin{proof}
Since $\phi^2=1$, the full sampled truncation (including its diagonal)
\[
F_{n,M}=(W_M(U_i,U_k))_{i,k\le n}
\]
obeys
\[
(F_{n,M}-n\lambda I)u_n=g_{n,M}
\]
exactly: the target term $\lambda\phi(U_i)\sum_k\phi(U_k)^2/\sqrt n=n\lambda u_{n,i}$ cancels. Put $R_M=W-W_M$. Coordinatewise,
\[
(g_n-g_{n,M})_i=n^{-1/2}\sum_{k\ne i}R_M(U_i,U_k)\phi(U_k)
 -n^{-1/2}W_M(U_i,U_i)\phi(U_i).
\]
For almost every $x$, $\int R_M(x,y)\phi(y)dy=0$ because $R_M\phi=0$. Conditional on $U_i=x$, the off-diagonal sum is centered and its summands are independent. Its conditional second moment is $(n-1)n^{-1}\int R_M(x,y)^2dy$, because $\phi(y)^2=1$, and its conditional cross product with the diagonal term is zero. Hence
\[
\E(g_n-g_{n,M})_i^2=(n-1)n^{-1}\norm{R_M}_{L^2}^2+n^{-1}\int W_M(x,x)^2dx.
\]
Sum over $i$, divide by $n$, and use the spectral Hilbert--Schmidt approximation $\norm{R_M}_{L^2}\to0$. The diagonal integral is finite for each fixed $M$ by boundedness of individual nonzero eigenfunctions. Markov's inequality gives \eqref{kr:eq:iter-tail}.
\end{proof}
\begin{lemma}[Quantified empirical fixed-mode resolvent approximation]\label{kr:lem:finite-mode}
For every fixed $j\ne r$,
\begin{align}
\E\norm{(M_n-\lambda_jI)e_{j,n}}^2 &=O_j(n^{-1}),\label{kr:eq:mode-residual}\\
\ip{u_n}{e_{j,n}}=\xi_{n,j}/\sqrt n&=O_p(n^{-1/2}),
\qquad \norm{Q_nM_nu_n}=O_p(n^{-1/2}).\label{kr:eq:mode-overlap}
\end{align}
On events where the compression spectral gap of Lemma~\ref{kr:lem:schur} holds,
\begin{equation}\label{kr:eq:mode-resolvent}
\norm{\mathcal R_ne_{j,n}-(\lambda-\lambda_j)^{-1}e_{j,n}}=o_p(1),
\end{equation}
and the stronger intermediate estimate is
\[
\norm{(Q_nM_nQ_n-\lambda_jI)Q_ne_{j,n}}=O_p(n^{-1/2}).
\]
For any finite set $F\subset J\setminus\{r\}$, the same claims hold jointly over $F$.
\end{lemma}
\begin{proof}
Given $U_i=x$, the $n-1$ terms of $(K_ne_{j,n})_i$ are conditionally independent, with mean $(n-1)\lambda_j\phi_j(x)/\sqrt n$ and variance at most $C_j$ because $W$ and $\phi_j$ are bounded. The difference from $n\lambda_j e_{j,n,i}$ has mean $-\lambda_j\phi_j(x)/\sqrt n$ and bounded conditional second moment. Summing over $i$ gives $\E\norm{(K_n-n\lambda_j I)e_{j,n}}^2\le C_jn$, proving \eqref{kr:eq:mode-residual}.
The overlap identity follows from the definition and the centered orthonormal coordinate $\psi_j$, and $Q_nM_nu_n=Q_ng_n/n$ proves the second assertion in \eqref{kr:eq:mode-overlap}. Set $\alpha_{j,n}=\ip{u_n}{e_{j,n}}$ and $z_{j,n}=Q_ne_{j,n}$. Directly, with $B_n=Q_nM_nQ_n|_{u_n^\perp}$,
\[
(B_n-\lambda_j I)z_{j,n}
=Q_n(M_n-\lambda_jI)e_{j,n}-\alpha_{j,n}Q_nM_nu_n.
\]
Both terms are $O_p(n^{-1/2})$ (the second is in fact $O_p(n^{-1})$). Therefore
\[
(\rho_nI-B_n)z_{j,n}
=(\rho_n-\lambda_j)z_{j,n}+O_p(n^{-1/2}).
\]
Apply the inverse, whose norm is $O_p(1)$ by Lemma~\ref{kr:lem:schur}, to obtain
\[
\mathcal R_n e_{j,n}
=(\rho_n-\lambda_j)^{-1}Q_ne_{j,n}+O_p(n^{-1/2}).
\]
Because $\rho_n\to_p\lambda$ and $\norm{Q_ne_{j,n}-e_{j,n}}=|\alpha_{j,n}|=O_p(n^{-1/2})$, the claim follows. The final assertion uses a union bound over finitely many indices. These estimates explicitly close the fixed-mode implication that cannot be inferred merely from an unquantified approximate-eigenvector assertion.
\end{proof}
\begin{lemma}[Convergence of the entire resolvent sandwich]\label{kr:lem:quad}
Jointly with any finite family $(\xi_{n,j})$,
\begin{equation}\label{kr:eq:quad-limit}
(g_n/\sqrt n)^{\mathsf T}\mathcal R_n(g_n/\sqrt n)
\ \Rightarrow\ \sum_{j\ne r}\frac{\lambda_j^2}{\lambda-\lambda_j}Z_j^2.
\end{equation}
The limiting series converges absolutely almost surely and in $L^1$.
\end{lemma}
\begin{proof}
For fixed $M$, Lemma~\ref{kr:lem:finite-mode}, empirical orthonormality $\ip{e_{j,n}}{e_{k,n}}\to_p\delta_{jk}$, and the finite-dimensional central limit theorem give
\[
(g_{n,M}/\sqrt n)^{\mathsf T}\mathcal R_n(g_{n,M}/\sqrt n)
\ \Rightarrow\ \sum_{j\in J_M\setminus\{r\}}
\frac{\lambda_j^2}{\lambda-\lambda_j}Z_j^2.
\]
The \emph{same} $\xi_{n,j}$ appear in every fixed-$M$ expansion. Put $x_n=g_n/\sqrt n$, $x_{n,M}=g_{n,M}/\sqrt n$. The deterministic inequality
\[
|\ip{x_n}{\mathcal R_nx_n}-\ip{x_{n,M}}{\mathcal R_nx_{n,M}}|
\le\norm{\mathcal R_n}_{\op}\norm{x_n-x_{n,M}}
(\norm{x_n}+\norm{x_{n,M}})
\]
combined with tightness of $\norm{\mathcal R_n}_{\op}$, $\sup_n\E\norm{x_n}^2<\infty$, and \eqref{kr:eq:iter-tail} makes the difference negligible in the iterated limit $n\to\infty$ followed by $M\to\infty$. Precisely, on the gap event choose a deterministic inverse bound $C_2$ and set
$A_n=\norm{x_n}$ and $\eta_{n,M}=\norm{x_n-x_{n,M}}$. The difference is at most
$C_2\eta_{n,M}(2A_n+\eta_{n,M})$, where $A_n$ is tight independently of $M$.
For $t>0$, the limsup probability of a difference greater than $\epsilon$ is
bounded by the tail probability of $\eta_{n,M}>t$ plus that of
$C_2t(2A_n+t)>\epsilon$, and the gap-failure probability. Send $n\to\infty$,
then $M\to\infty$, and finally $t\downarrow0$ to obtain zero. This also works
for the vector of the Rayleigh statistic and the quadratic form. Finally, by isolation,
\[
\sum_{j\ne r}\frac{\lambda_j^2}{|\lambda-\lambda_j|}
\le\gamma^{-1}\sum_j\lambda_j^2<\infty,
\]
so the truncation limits converge in $L^1$ and $L^2$ (their means have absolutely summable tails and their centered variances have square-summable tails), and the converging-together principle establishes \eqref{kr:eq:quad-limit} and joint convergence.
\end{proof}
\subsection{Kato--Temple and deterministic perturbation bounds}
\begin{lemma}[Kato--Temple inequality]\label{lem:kato-temple}
Let $M$ be symmetric and $u$ a unit vector with $\eta=u^{\mathsf T}Mu$ and $\epsilon=\|Mu-\eta u\|$. If $\alpha<\eta<\beta$ and $(\alpha,\beta)$ contains exactly one eigenvalue $\theta$, then
\[
\eta-\frac{\epsilon^2}{\beta-\eta}\le\theta
\le\eta+\frac{\epsilon^2}{\eta-\alpha}.
\]
\end{lemma}
\begin{proof}
The spectral exclusion makes $(M-\theta I)(M-\beta I)$ and $(M-\alpha I)(M-\theta I)$ positive semidefinite. Their quadratic forms at $u$ are
$\epsilon^2+(\eta-\theta)(\eta-\beta)$ and
$\epsilon^2+(\eta-\alpha)(\eta-\theta)$, respectively. Rearranging gives the two bounds.
\end{proof}
\begin{lemma}[First- and second-order perturbation bounds]\label{lem:expansion}
Let $M$ be real symmetric, $\theta$ a simple eigenvalue with gap $g>0$ and
unit eigenvector $v$. Let $E$ be symmetric with $e=\norm E_{\op}\le g/4$.
The matching perturbed eigenvalue $\widehat\theta$ is unique near $\theta$.
On $v^\perp$ set $Q=I-vv^{\mathsf T}$ and
$R=(\theta I-QMQ)^{-1}$, extended by zero on $v$. Then
\[
\widehat\theta-\theta=v^{\mathsf T}Ev+(Ev)^{\mathsf T}R(Ev)+r,
\qquad |r|\le4e^3/g^2,
\]
and $|\widehat\theta-\theta-v^{\mathsf T}Ev|\le2e^2/g$.
\end{lemma}
\begin{proof}
Weyl gives $|\widehat\theta-\theta|\le e$ and isolates the matching branch.
The distance of $\widehat\theta$ from the spectrum of $Q(M+E)Q|_{v^\perp}$
is at least $g-2e\ge g/2$. Thus its inverse norm is at most $2/g$.
The exact Schur formula gives the first-order bound. The inverse at the
unperturbed center has norm at most $1/g$, and the difference of the two
matrices has norm at most $2e$. Their inverse difference is at most
$4e/g^2$. Multiplication by $\norm{QEv}^2\le e^2$ gives the cubic remainder.
\end{proof}

\section*{Appendix B: U-Statistic and Kernel Decompositions} \setcounter{section}{2} \renewcommand{\thesection}{B.\arabic{section}} \renewcommand{\thelemma}{B.\arabic{lemma}}

\begin{lemma}[Canonical Rayleigh truncation]\label{lem:degU}
Under $\phi=\varphi_r$ and $\phi^2=1$, let $h_r(x,y)=\phi(x)W(x,y)\phi(y)$ and $h_{r,2}=h_r-\lambda_r$. Then $h_{r,2}$ is bounded, symmetric and canonical. For finite mode sets $J_M\ni r$, define $h_M=\sum_{k\in J_M\setminus\{r\}}\lambda_k\psi_k\otimes\psi_k$ with $\psi_k=\phi\varphi_k$. Exactly,
\[
\frac1n\sum_{i\ne j}h_M(U_i,U_j)
=\sum_{k\in J_M\setminus\{r\}}\lambda_k
\left(\xi_{n,k}^2-\frac1n\sum_i\psi_k(U_i)^2\right),
\]
and
\[
\E\left[\left(\frac1n\sum_{i\ne j}(h_{r,2}-h_M)(U_i,U_j)\right)^2\right]
=2(1-1/n)\|W-W_M\|_{L^2}^2.
\]
\end{lemma}
\begin{proof}
The eigenrelation gives $\int h_r(x,y)dy=\lambda_r\phi(x)^2=\lambda_r$ and $\E h_r(U,U')=\lambda_r$. Multiplication by $\phi$ preserves orthonormality. The finite sum identity follows by expanding the square and removing its diagonal. In the canonical tail variance, only identical and reversed unordered index pairs survive; orthonormality gives its $L^2$ norm as $\|W-W_M\|_{L^2}$. This proves the exact bound. It licenses the iterated finite-mode passage, without an unjustified infinite empirical diagonal expansion.
\end{proof}

\begin{lemma}[Square-summability of fluctuation coefficients]\label{lem:squaresum}
Let \(T_W:L^2([0,1])\to L^2([0,1])\) be the compact self-adjoint operator associated with the symmetric kernel \(W\in L^2([0,1]^2)\cap L^\infty([0,1]^2)\), and let \(\{(\lambda_k,\varphi_k)\}_{k\ge1}\) be its orthonormal eigenpairs. 
Fix a simple eigenvalue \(\lambda_r\) with eigengap 
\(\gamma_r=\dist(\lambda_r,\spec(T_W)\setminus\{\lambda_r\})>0\).
Then
\[
\sum_{k\neq r} \left( \frac{\lambda_r\lambda_k}{\lambda_r - \lambda_k} \right)^2 < \infty.
\]
\end{lemma}

\begin{proof}
Since \(T_W\) is Hilbert-Schmidt, we have 
\(\sum_{k=1}^\infty \lambda_k^2 = \|W\|_{L^2([0,1]^2)}^2 < \infty.\)
For \(k\neq r\), note that \(|\lambda_r - \lambda_k| \ge \gamma_r > 0\) by the eigengap assumption, so
\[
\left(\frac{\lambda_r\lambda_k}{\lambda_r - \lambda_k}\right)^2 
\le \frac{\lambda_r^2}{\gamma_r^2}\,\lambda_k^2.
\]
Summing over all \(k\neq r\) yields
\[
\sum_{k\neq r}\left(\frac{\lambda_r\lambda_k}{\lambda_r - \lambda_k}\right)^2
\le \frac{\lambda_r^2}{\gamma_r^2} \sum_{k\neq r}\lambda_k^2
\le \frac{\lambda_r^2}{\gamma_r^2}\|W\|_{L^2([0,1]^2)}^2 < \infty.
\]
Hence the coefficients are square-summable, and the random series
\(\sum_{k\neq r} \tfrac{\lambda_r\lambda_k}{\lambda_r - \lambda_k}(Z_k^2-1)\)
converges in \(L^2\) by the completeness of \(\ell^2\) and independence of the Gaussian sequence \((Z_k)\).
\end{proof}

\begin{lemma}\label{lem:norm}
Let
\[
V_{r,n}:=\frac{1}{n}\sum_{i=1}^n(\varphi_r(U_i)^2-1),
\qquad
s_n:=\sum_{i=1}^n \varphi_r(U_i)^2 = n(1+V_{r,n}).
\]
Assume $\Var(\varphi_r(U)^2)<\infty$. Then $V_{r,n}=O_p(n^{-1/2})$ and
\[
\frac{1}{s_n}
=
\frac{1}{n}\Bigl(1 - V_{r,n} + O_p(V_{r,n}^2)\Bigr),
\qquad
\frac{1}{\sqrt{s_n}}
=
\frac{1}{\sqrt{n}}\Bigl(1 - \tfrac{1}{2}V_{r,n} + O_p(V_{r,n}^2)\Bigr).
\]
In particular, since $V_{r,n}^2=O_p(n^{-1})$, both remainders are $O_p(n^{-1})$.
\end{lemma}

\begin{proof}
Assume $\Var(\varphi_r(U)^2)<\infty$ and recall
\[
V_{r,n}:=\frac{1}{n}\sum_{i=1}^n\bigl(\varphi_r(U_i)^2-1\bigr),
\qquad
s_n:=\sum_{i=1}^n \varphi_r(U_i)^2=n(1+V_{r,n}).
\]
Since $\E[\varphi_r(U)^2]=1$ and $\Var(\varphi_r(U)^2)<\infty$, we have
\[
\E[V_{r,n}]=0,
\qquad
\Var(V_{r,n})=\frac{1}{n}\Var(\varphi_r(U)^2),
\]
hence by Chebyshev's inequality $V_{r,n}=O_p(n^{-1/2})$.

For the first expansion, write
\[
\frac{1}{s_n}=\frac{1}{n}(1+V_{r,n})^{-1}.
\]
Using the Taylor expansion of $f(x)=(1+x)^{-1}$ at $0$,
\[
(1+x)^{-1}=1-x+x^2+R_1(x),
\]
where the remainder satisfies $|R_1(x)|\le C|x|^3$ for all $|x|\le 1/2$ and a
universal constant $C>0$. On the event $\{|V_{r,n}|\le 1/2\}$,
\[
(1+V_{r,n})^{-1}
=
1-V_{r,n}+V_{r,n}^2+R_1(V_{r,n}),
\qquad
|R_1(V_{r,n})|\le C|V_{r,n}|^3.
\]
Since $V_{r,n}=O_p(n^{-1/2})$, we have $V_{r,n}^2=O_p(n^{-1})$ and
$V_{r,n}^3=O_p(n^{-3/2})$, so
\[
(1+V_{r,n})^{-1}
=
1-V_{r,n}+O_p(V_{r,n}^2).
\]
Multiplying by $1/n$ yields
\[
\frac{1}{s_n}
=
\frac{1}{n}\Bigl(1-V_{r,n}+O_p(V_{r,n}^2)\Bigr).
\]

For the second expansion, similarly
\[
\frac{1}{\sqrt{s_n}}=\frac{1}{\sqrt n}(1+V_{r,n})^{-1/2}.
\]
Using the Taylor expansion of $g(x)=(1+x)^{-1/2}$ at $0$,
\[
(1+x)^{-1/2}=1-\tfrac12 x+R_2(x),
\]
where $|R_2(x)|\le C'x^2$ for all $|x|\le 1/2$ and a universal constant $C'>0$.
On $\{|V_{r,n}|\le 1/2\}$ we obtain
\[
(1+V_{r,n})^{-1/2}
=
1-\tfrac12 V_{r,n}+R_2(V_{r,n}),
\qquad
|R_2(V_{r,n})|\le C'V_{r,n}^2,
\]
hence
\[
(1+V_{r,n})^{-1/2}
=
1-\tfrac12 V_{r,n}+O_p(V_{r,n}^2).
\]
Multiplying by $1/\sqrt n$ gives
\[
\frac{1}{\sqrt{s_n}}
=
\frac{1}{\sqrt n}\Bigl(1-\tfrac12 V_{r,n}+O_p(V_{r,n}^2)\Bigr).
\]
Since $V_{r,n}^2=O_p(n^{-1})$, both $O_p(V_{r,n}^2)$ remainders are also
$O_p(n^{-1})$.
\end{proof}

\section*{Appendix C: Order-one edge-noise correction}
\renewcommand{\thelemma}{C.\arabic{lemma}}
\setcounter{subsection}{0}
\renewcommand{\thesubsection}{C.\arabic{subsection}}
Throughout this appendix use the notation of Theorem~\ref{er:thm:edge}: $\lambda=\lambda_r$, $\phi=\varphi_r$, $\gamma=\gamma_r$, $K_n=\mathbf K_n$, $A_n=\mathbf A_n$, $E_n=A_n-K_n$, and $u_n=n^{-1/2}(\phi(U_i))_i$. Here $\phi^2$ need not be constant. Let $t_n,\widehat t_n$ be the matching isolated eigenvalue branches and $v_n$ a unit eigenvector for $t_n$. The constants $s,\tau,\nu$ are defined in that theorem.
\subsection{Inputs, alignment, and the only matrix-concentration input}
The localization lemma, Lemma~\ref{kr:lem:localize}, is proved without $\phi^2=1$. It gives
\[
t_n/n\to_p\lambda,\quad
\min_{\mu\ne t_n}|t_n-\mu|\ge\gamma n/2,
\quad \norm{v_n-u_n}=O_p(n^{-1/2}),\quad
\norm{v_n}_\infty=O_p(n^{-1/2}),
\]
where $u_n=n^{-1/2}(\phi(U_i))_i$ and $\norm{u_n}^2=1+O_p(n^{-1/2})$.
\begin{lemma}[Uniform conditional bounds]\label{er:lem:bounds}
For vectors $a,b$ measurable with respect to $\mathcal U_n$,
\[
\Var(a^{\mathsf T}E_nb\mid\mathcal U_n)
=\sum_{i<j}s(U_i,U_j)(a_ib_j+a_jb_i)^2
\le\tfrac12\norm a^2\norm b^2.
\]
Moreover $\norm{E_n}_{\op}=O_p(\sqrt{n\log n})=o_p(n)$, uniformly over
the conditional edge probabilities.
\end{lemma}
\begin{proof}
The variance formula is direct independence. Apply $(x+y)^2\le2x^2+2y^2$
and $s\le1/4$ to obtain the bound. Write
$E_n=\sum_{i<j}E_{ij}(e_ie_j^{\mathsf T}+e_je_i^{\mathsf T})$. Each summand is
conditionally centered, self-adjoint, and has operator norm at most one.
The variance proxy is the diagonal matrix with entries
$\sum_{j\ne i}s(U_i,U_j)\le(n-1)/4$. Apply the matrix Bernstein inequality of \citet[Theorem 1.4]{tropp2012user} to $E_n$ and $-E_n$:
\[
\Prob(\norm{E_n}_{\op}\ge x\mid\mathcal U_n)
\le2n\exp\left\{-\frac{x^2}{2((n-1)/4+x/3)}\right\}.
\]
At $x=2\sqrt{n\log n}$ the right-hand side tends to zero. The bound is
uniform in the latent sample. We retain the logarithm at every subsequent
use; a sharper square-root bound is unnecessary.
\end{proof}

\subsection{The linear Gaussian term}
\begin{lemma}\label{er:lem:linear}
For every real $t$,
\[
\E[e^{itv_n^{\mathsf T}E_nv_n}\mid\mathcal U_n]\to_p e^{-t^2\nu/2}.
\]
\end{lemma}
\begin{proof}
Put $M=\norm{\phi}_\infty<\infty$ and
$L_n=u_n^{\mathsf T}E_nu_n=\sum_{i<j}X_{ij}$, where
$X_{ij}=2u_{n,i}u_{n,j}(E_n)_{ij}$. Conditionally on $\mathcal U_n$,
these summands are independent and centered, and
\[
 |X_{ij}|\le 2M^2/n=:b_n
\]
almost surely. Thus the conditional Lindeberg sum is identically zero
for $n>2M^2/\varepsilon$, outside a single null set arising from the
chosen bounded representative of $\phi$. Their total conditional variance is
\[
 V_n=4n^{-2}\sum_{i<j}\phi(U_i)^2\phi(U_j)^2s(U_i,U_j)\to_p\nu
\]
by the bounded U-statistic law of large numbers; also $V_n\le M^4/2$
deterministically.

For completeness, this gives a uniform characteristic-function estimate,
without a pathwise assertion about convergence of $V_n$.
Taylor expansion gives
\[
 \left|\E[e^{itX_{ij}}\mid\mathcal U_n]
       -1+\tfrac12t^2\E[X_{ij}^2\mid\mathcal U_n]\right|
 \le\tfrac16|t|^3b_n\E[X_{ij}^2\mid\mathcal U_n].
\]
Compare each factor with $\exp(-t^2\E[X_{ij}^2\mid\mathcal U_n]/2)$.
Since both factors have modulus at most one, telescoping their products
bounds the total error by
$|t|^3b_nV_n/6+t^4b_n^2V_n/8=O_t(n^{-1})$.
Consequently
\[
 \E[e^{itL_n}\mid\mathcal U_n]-e^{-t^2V_n/2}\to0
\]
uniformly over the latent sample, and the conditional characteristic
function converges in probability to $e^{-t^2\nu/2}$, including when $\nu=0$.

It remains to transfer this result to $v_n$. The difference is exactly
\[
 D_n=v_n^{\mathsf T}E_nv_n-u_n^{\mathsf T}E_nu_n
 =(v_n-u_n)^{\mathsf T}E_n(v_n+u_n).
\]
Lemma~\ref{er:lem:bounds} and the alignment estimate imply
\[
 \E[D_n^2\mid\mathcal U_n]
 \le\tfrac12\norm{v_n-u_n}^2\norm{v_n+u_n}^2=O_p(n^{-1}).
\]
Thus conditional Chebyshev gives
\[
\Prob(|D_n|>\varepsilon)
\le\E\min\left(1,\frac{\E[D_n^2\mid\mathcal U_n]}{\varepsilon^2}\right)\to0,
\]
since the integrand is bounded and converges to zero in probability.
Moreover the difference of the two conditional characteristic functions
is bounded by
\[
|t|\sqrt{\E[D_n^2\mid\mathcal U_n]}=o_p(1).
\]
This proves the claimed conditional limit directly.
\end{proof}

\subsection{Quadratic energy and the bulk resolvent}
\begin{lemma}\label{er:lem:energy}
\[
\norm{E_nv_n}^2/n\to_p\tau,\qquad
\Var(\norm{E_nv_n}^2\mid\mathcal U_n)\le2(n-1).
\]
\end{lemma}
\begin{proof}
Its conditional mean divided by $n$ is
$\sum_jv_{n,j}^2\{n^{-1}\sum_{i\ne j}s(U_i,U_j)\}$.
Replace $v_{n,j}^2$ by $u_{n,j}^2$; the error is bounded by
$\tfrac14\sum_j|v_{n,j}^2-u_{n,j}^2|=o_p(1)$.
The remaining expression is $n^{-2}\sum_{i\ne j}\phi(U_j)^2s(U_i,U_j)$,
which converges to $\tau$ by the bounded U-statistic LLN (symmetrize the kernel).

Index the unordered edges by $e=\{i,j\}$, set $\varepsilon_e=E_{ij}$,
and let $B$ have edge column $v_{n,j}e_i+v_{n,i}e_j$. Then
$E_nv_n=B\varepsilon$ and, for every unit vector $v_n$,
\[
BB^{\mathsf T}=I-2\diag(v_{n,i}^2)+v_nv_n^{\mathsf T},\qquad
\tr[(BB^{\mathsf T})^2]=n-1.
\]
The last equality follows by summing the diagonal squares
$\sum_i(1-v_{n,i}^2)^2=n-2+\sum_i v_{n,i}^4$ and the off-diagonal squares
$\sum_{i\ne j}v_{n,i}^2v_{n,j}^2=1-\sum_i v_{n,i}^4$.
Put $M=B^{\mathsf T}B$. For independent centered bounded edge variables,
\[
\Var(\varepsilon^{\mathsf T}M\varepsilon\mid\mathcal U_n)
\le\sum_eM_{ee}^2+2\sum_{e\ne f}M_{ef}^2
\le2\tr(M^2)=2(n-1).
\]
Indeed, a diagonal term and an off-diagonal term have zero covariance, and
two distinct off-diagonal terms have zero covariance unless they use the same
unordered pair of edge indices. Conditional Chebyshev finishes the proof.
\end{proof}

Let $Q=I-v_nv_n^{\mathsf T}$. On $v_n^\perp$ define
$R_n=(t_nI-QK_nQ)^{-1}$ and extend it by zero on $v_n$.
This compression is invariant because $v_n$ is a true eigenvector, so
$\norm{R_n}_{\op}\le2/(\gamma n)$ on the gap event.

\begin{lemma}\label{er:lem:bulk}
$(E_nv_n)^{\mathsf T}R_n(E_nv_n)\to_p\tau/\lambda$.
\end{lemma}
\begin{proof}
Fix $\delta\in(0,|\lambda|/4)$. Let $H_{\delta,n}$ be the spectral
projection of $K_n$ in $v_n^\perp$ onto eigenvalues $|\mu|>n\delta$.
Since $\norm{K_n}_F^2\le n(n-1)$, its rank is at most $\delta^{-2}$.
For a conditionally measurable orthonormal basis $w_{\ell,n}$ of its range,
Lemma~\ref{er:lem:bounds} gives
\[
\E[\norm{H_{\delta,n}E_nv_n}^2\mid\mathcal U_n]
=\sum_\ell\E[(w_{\ell,n}^{\mathsf T}E_nv_n)^2\mid\mathcal U_n]
\le\tfrac12\operatorname{rank}(H_{\delta,n})\le\frac1{2\delta^2}.
\]
Consequently $\norm{H_{\delta,n}E_nv_n}^2=O_p(\delta^{-2})$.
The same conditional variance bound gives
$(v_n^{\mathsf T}E_nv_n)^2=O_p(1)$. Put
\[
z_{\delta,n}=(I-v_nv_n^{\mathsf T}-H_{\delta,n})E_nv_n.
\]
These three components are orthogonal. For each fixed $\delta$,
Lemma~\ref{er:lem:energy} therefore yields
\[
\frac{\norm{z_{\delta,n}}^2}{n}
=\frac{\norm{E_nv_n}^2}{n}
-\frac{\norm{H_{\delta,n}E_nv_n}^2}{n}
-\frac{(v_n^{\mathsf T}E_nv_n)^2}{n}
=\frac{\norm{E_nv_n}^2}{n}+o_p(1)\to_p\tau.
\]
Write $\rho_n=|t_n/n-\lambda|\to_p0$. On the gap event, each eigenvalue
supporting $z_{\delta,n}$ satisfies $|\mu|\le n\delta$ and
\[
\left|\frac1{t_n-\mu}-\frac1{n\lambda}\right|
\le\frac{2(\rho_n+\delta)}{\gamma|\lambda|n}.
\]
Hence, with $C=2/(\gamma|\lambda|)$,
\[
\left|z_{\delta,n}^{\mathsf T}R_nz_{\delta,n}
-\frac{\norm{z_{\delta,n}}^2}{n\lambda}\right|
\le C(\rho_n+\delta)\frac{\norm{z_{\delta,n}}^2}{n}.
\]
The projections commute with $R_n$, which is zero on $v_n$; their cross
terms thus vanish exactly. In particular,
\[
(E_nv_n)^{\mathsf T}R_n(E_nv_n)
=(H_{\delta,n}E_nv_n)^{\mathsf T}R_n(H_{\delta,n}E_nv_n)
+z_{\delta,n}^{\mathsf T}R_nz_{\delta,n}.
\]
The high-mode term has absolute value at most
$2\norm{H_{\delta,n}E_nv_n}^2/(\gamma n)
=O_p(\delta^{-2}/n)=o_p(1)$ for fixed $\delta$.
Finally set $B_n=\norm{E_nv_n}^2/n\to_p\tau$.
Since $\norm{z_{\delta,n}}^2/n\le B_n$, the discrepancy from
$\norm{z_{\delta,n}}^2/(n\lambda)$ is bounded by
$C\delta B_n+o_p(1)$ for each fixed $\delta$.
For every $\varepsilon>0$, first take the limsup in $n$ of the probability
that the discrepancy from $\tau/\lambda$ exceeds $\varepsilon$, and then
send $\delta\downarrow0$. Tightness of $B_n$, independent of $\delta$,
makes this probability zero. This proves the claim for either sign of $\lambda$.
\end{proof}

\subsection{Schur expansion and independence}
\begin{proof}[Proof of Theorem~\ref{er:thm:edge}]
Weyl's inequality and Lemma~\ref{er:lem:bounds} give a unique matching branch,
with $|\widehat t_n-t_n|\le\norm{E_n}_{\op}=O_p(\sqrt{n\log n})$.
On $v_n^\perp$, the distance of $\widehat t_n$ from the spectrum of $QA_nQ$
is at least $\gamma n/2-2\norm{E_n}_{\op}\ge\gamma n/4$ with probability
tending to one. The exact block eigenvalue equation then gives
\[
\widehat t_n-t_n=v_n^{\mathsf T}E_nv_n+
(QE_nv_n)^{\mathsf T}(\widehat t_nI-QA_nQ)^{-1}(QE_nv_n).
\]
The inverse identity, on that same subspace, yields
\[
\norm{(\widehat t_nI-QA_nQ)^{-1}-(t_nI-QK_nQ)^{-1}}_{\op}
=O_p\left(\frac{\sqrt{n\log n}}{n^2}\right).
\]
Multiply by $\norm{QE_nv_n}^2=O_p(n)$ from Lemma~\ref{er:lem:energy}.
The error is $O_p(\sqrt{\log n/n})=o_p(1)$; no sharper concentration bound
is invoked. Lemma~\ref{er:lem:bulk} proves \eqref{er:eq:edge-expansion}.

If the resulting remainder is $r_n=o_p(1)$, then
$\E\min(2,|t r_n|)\to0$. Markov's inequality and conditioning imply
$\E[\min(2,|t r_n|)\mid\mathcal U_n]\to_p0$.
Combine this with Lemma~\ref{er:lem:linear} to prove \eqref{er:eq:conditional}.
For a latent statistic $H_n\Rightarrow H$, condition on $\mathcal U_n$ in
$\E[e^{iaH_n+it(\widehat t_n-t_n)}]$. The bounded conditional characteristic
functions converge in probability to a deterministic value, hence in $L^1$.
The limit factors as
$\E[e^{iaH}]\exp(it\tau/\lambda-t^2\nu/2)$, proving the asserted joint law.
This establishes independence rather than inferring it from uncorrelated
finite-sample diagnostics.
\end{proof}

\begingroup
\small
\bibliographystyle{plainnat}
\bibliography{Reference}
\endgroup
\end{document}